\documentclass[12pt]{amsart}
\newtheorem{theorem}{Theorem}[section]
\newtheorem{lemma}[theorem]{Lemma}
\newtheorem{proposition}[theorem]{Proposition}

 \theoremstyle{definition}

\newtheorem{problem}[theorem]{Problem}

\theoremstyle{remark}
\newtheorem{remark}[theorem]{Remark}

\numberwithin{equation}{section}

\begin{document}
\setlength{\baselineskip}{1.2\baselineskip}

\title[Pogorelov type estimates]
{Pogorelov type estimates for $(n-1)$-Hessian equations and 
related rigidity theorems}

\author{Qiang Tu}
\address{Faculty of Mathematics and Statistics, Hubei Key Laboratory of Applied
Mathematics, Hubei University,  Wuhan 430062, P.R. China}
\thanks{Research of the author was supported by funds from the National
Natural Science Foundation of China No. 12101206; the Natural Science Foundation of Hubei Province, China, No. 2023AFB730.}

\email{qiangtu@hubu.edu.cn}

\begin{abstract}
In this paper, we establish  Pogorelov type $C^2$ estimates for admissible solutions to the Dirichlet problem of $(n-1)$-Hessian equation  based on a concavity inequality, which is inspired by the Lu-Tsai's work on the global curvature estimates for the $n-1$ curvature equation. As an application, we apply such estimates to obtain a rigidity theorems for admissible solutions of $(n-1)$-Hessian equation only under quadratic growth conditions. This result gives a positive answer to a open problem for $k$-Hessian equation, which is  proposed by Chang-Yuan, in case $k=n-1$.

{\em Mathematical Subject Classification (2010):}
 Primary 35J60, Secondary 35B45.

{\em Keywords:} Pogorelov type estimates, Hessian equations, Rigidity theorem.

\end{abstract}

\maketitle
\bigskip


\section{Introduction}

\medskip

Let $\Omega \subset \mathbb{R}^n$
be a bounded domain with smooth boundary and $n\geq 3$. In this paper we consider a Pogorelov type $C^2$ estimates for the following Dirichlet problem
of Hessian equations for $k=n-1$
\begin{equation}\label{Eq}
\left\{
\begin{aligned}
&\sigma_{k}(\nabla^2u)=\psi(x,u,\nabla u) &&in~
\Omega,\\
&u = 0 &&on~\partial \Omega,
\end{aligned}
\right.
\end{equation}
where $\sigma_k (\nabla^2u)$ denote by $\sigma_k
(\lambda(\nabla^2u))$ (see \eqref{g}) with $\lambda(\nabla^2u)$
being the eigenvalues of the Hessian matrix $\nabla^2u$, $\psi$ is a positive and smooth function in $\Omega\times \mathbb{R}\times \mathbb{R}^n$. Moreover,
$u\in C^2(\Omega)$ is admissible if $\lambda(\nabla^2u)$ belongs to
the Garding's cone $\Gamma_k$ (see \eqref{g-k}).

As we all know, the classic $k$-Hessian equation 
\begin{equation}\label{Eq-0p}
\sigma_k(\nabla^2u)=\psi(x,u,\nabla u)
\end{equation}
is an important research content in the field of fully nonlinear partial differential equations and geometric analysis, which is associated with many important geometric problems, such as Minkowski problem, prescribing curvature problem and so on, see \cite{Chu21, Chu20, Gi98, Ger03, Guan-12, Guan-09, GRW15, lu23, RW-19, RW-20, RW-200, Sheng07, Sz17, Sz18, To17, To19, Wu87}. The a prior estimates, especially the $C^2$ estimates,  for $k$-Hessian equation \eqref{Eq-0p} is a longstanding problem, which has attracted much attentions. For the relative work, we refer the readers to \cite{Chen20, CNS84, CNS85, Chou01,  Ha09, Li04, Li96, S05, Tr95, TW-08, Wang09}.

In the sequel we study Pogorelov type $C^2$ estimates, which is a type of interior $C^2$ estimates with boundary information, for the Dirichlet problem \eqref{Eq}. The initial motivation of our work is the following: Pogorelov estimates for Monge-Amp\`ere equations were studied at first by Pogorelov \cite{Gi98, Po78}.
Then, Chou-Wang extended
Pogorelov type estimates to the case of $k$-Hessian equations
\cite{Chou01, Wang09}. More precisely,  when the right hand side function $\psi=\psi(x,u)$,  Chou-Wang proved
that there exist a constant $\epsilon>0$ such that 
\begin{equation}\label{Pog}
\sup_{\Omega}(-u)^{1+\epsilon}|\nabla^2 u|\leq C
\end{equation}
for any $k$-convex solution $u$ to the Hessian equations
\eqref{Eq}. 
It's worth pointing out that the small constant $\epsilon$ should not be zero in Chou-Wang's proof.
Later, Li-Ren-Wang \cite{Li16} developed new techniques to drop the small $\epsilon$ and established Pogorelov
type estimates for $(k+1)$-convex solutions to the Hessian equations \eqref{Eq} with $\psi$ depending on $x, u$ and the gradient
term $\nabla u$. Recently, Zhang \cite{Zhang24} derived a concavity inequality for $k$-Hessian operators under the semi-convexity condition. As an application, he established Pogorelov type $C^2$ estimates for semi-convex and admissible solution to the  Hessian equations \eqref{Eq}.

From analysis point of view, a natural problem is whether  we can weak the semi-convexity assumption in \cite{Zhang24} and established Pogorelov type $C^2$ estimates to the  Hessian equations \eqref{Eq}.
 Our idea goes back as far as Ren-Wang's work \cite{RW-19, RW-20, RW-200}  on the global curvature estimates for the $n-1$ curvature equation. More precisely, they established the global curvature estimate for the following prescribing curvature equation 
\begin{equation}\label{008989}
\sigma_{n-1}(\kappa(X))=f(X, \nu(X)), \quad \forall~X\in M,
\end{equation}
$M\subset \mathbb{R}^{n+1}$ is a closed hypersurface, $\kappa(X)$ and $\nu(X)$ are principal curvatures and unit outer normal vector at $X\in M$, respectively. Note that, their results based on  conjecture inequality, and the conjecture holds when $k\geq n-2$. Recently,  Lu-Tsai\cite{lu24} gave a simple proof of the curvature estimates for the  equation \eqref{008989}. The main idea of their proof is to separate the arguments into semi-convex and non-semi-convex cases. 

Inspired by these results, we establish the following Pogorelov type $C^2$ estimates for admissible solutions of Hessian equation \eqref{Eq} for $k=n-1$.

\begin{theorem}\label{main-2}
Let $n>2$ and $\psi \in C^{1,1}(\overline{\Omega}\times \mathbb{R}\times \mathbb{R}^n)$ with $\psi>0$. Assume $u\in C^4(\Omega)\cap C^2(\overline{\Omega})$ is a solution to the following Dirichlet problem of $(n-1)$-Hessian equations 
\begin{equation}\label{Eq-n-1}
\left\{
\begin{aligned}
&\sigma_{n-1}(\nabla^2u)=\psi(x,u,\nabla u) &&in~
\Omega,\\
&u = 0 &&on~\partial \Omega,
\end{aligned}
\right.
\end{equation}
 with $\lambda(\nabla^2u) \in
\Gamma_{n-1}$. Then there exist a constant $\beta>0$ such that
\begin{equation*}
(-u)^{\beta}|\nabla^2 u|(x) \leq C, \quad \forall~ x \in \Omega,
\end{equation*}
where $C, \beta$ depends on $n, \sup_{\Omega} |u|, \sup_{\Omega} |\nabla u|, |\psi|_{C^2}$, $\inf_{\Omega}\psi$..
\end{theorem}

As an application of the above results, a rigidity theorem can be considered  for $(n-1)$-Hessian equations. In \cite{chang-10}, Chang-Yuan  proposed a problem that:
\begin{problem}\label{pro-semi-1} 
Are the entire solutions of the following $k$-Hessian equations 
\begin{equation}\label{digity-00}
\sigma_k( \nabla^2 u(x))=1\quad \forall~x\in\mathbb{R}^n,
\end{equation}
 with lower bound only quadratic polynomials?
\end{problem}

Recall that the classical Liouville theorem for Laplace equations and J$\ddot{o}$rgens-Calabi-Pogorelov theorem for Monge-Amp\'ere equations gave a positive answer to the  Problem \ref{pro-semi-1}  in case  $k=1$ and $k=n$, respectively.  For $k=2$, Chang-Yuan \cite{chang-10}  have proved that, if $u$ satisfies an almost convexity conditon, i.e., 
$$\nabla^2 u \geq \left(\delta-\sqrt{\frac{2n}{n-1}} \right)\quad \forall~\delta>0,$$
then the entire solution of the equation  \eqref{digity-00} only are quadratic polynomials. Chen-Xiang \cite{Chen19} showed that all “super quadratic” entire
solutions to equation  \eqref{digity-00} with $\sigma_1( \nabla^2 u) >0$ and $\sigma_3( \nabla^2 u)\geq -K$ are also quadratic polynomials.
 Then Shankar-Yuan \cite{RY-22} have proved that every entire semi-convex solution of equation \eqref{digity-00} in three dimensions is a quadratic polynomial. For general $k$, Bao-Chen-Guan-Ji \cite{Bao-03} proved that strictly entire convex solutions of equation \eqref{digity-00} satisfying a quadratic growth are quadratic polynomials. Li-Ren-Wang \cite{Li16} relaxed the condition of strictly convex solutions in \cite{Bao-03} to $(k+1)$-convex solutions. It was guessed in\cite{chang-10}  and \cite{Warren-16} that the semi-convex assumption is the necessary condition for the Problem \ref{pro-semi-1}. This conjecture was recently confirmed by Zhang in \cite{Zhang24}, in which he solved the Problem \ref{pro-semi-1} under  semi-convex assumption. However, it is to be expected that the Problem \ref{pro-semi-1} is still true without  semi-convex assumption in some  special cases. 
In this paper, we confirm Chang-Yuan's problem in case $k=n-1$.
 The following is our result.

\begin{theorem}\label{main-4}
The entire and admissible solutions  of the equation 
$$\sigma_{n-1} (D^2 u)=1$$
 defined in $\mathbb{R}^n$ satisfy the quadratic growth condition: there exist positive constants $c, b$ and sufficiently large $R$ such that
\begin{eqnarray}\label{rigidity-03}
u(x) \geq c|x|^2-b\quad \mbox{for}~|x|\geq R.
\end{eqnarray}
Then $u$ are quadratic polynomials.
\end{theorem}

It is worth noting that whether in establishing the pogorelov type $C^2$ estimates and deriving the rigidity theorem, the main difficulty to overcome is how to deal with the third order derivatives. For this purpose, we establish the following concavity inequality of the Hessian operator for $k=n-1$.

\begin{lemma}\label{pro-55-33}
Let $n\geq 3$, let $F=\sigma_{n-1}$ and $\lambda=(\lambda_1, \cdots, \lambda_n)\in \Gamma_{n-1}$ with
$$\lambda_1=\cdots=\lambda_m>\lambda_{m+1} \geq \cdots \geq \lambda_n.$$
Then there exists  constant $C>1$ depending on $n, F$ and small constant $\delta_0>0$ depending on depending on $n$ and $\max F$ such that if $\lambda_1 \geq C$, then  the following inequlity holds
\begin{eqnarray}\label{pro-55443-33}
-\sum_{p\neq q} F^{pp, qq} \xi_{p} \xi_{q}+ K\frac{\left(\sum_i F^{ii} \xi_{i}\right)^2}{F}+2 \sum_{i>m} \frac{F^{ii} \xi_{i}^2}{\lambda_1-\lambda_i} \geq (1+\delta_0) \frac{ F^{11}\xi_{1}^2}{\lambda_1}
\end{eqnarray}
for some sufficient large $K>0$ (depending on $\delta_0$), where $\xi=(\xi_1, \cdots, \xi_n)\in \mathbb{R}^n$ with $\xi_i=0$ for $1<i \leq m$.
\end{lemma}

\begin{remark}
The main idea of our proof is to separate the argument into case $\lambda_n> -A$ and $\lambda_n\leq -A$ ($A>0$) referring to  Zhang \cite{Zhang24} and  Lu-Tsai's work \cite{lu24}, respectively.  In case $\lambda_n>-A$, we 
 optimize the third term of   concavity inequalities established by Zhang so that bad third-order terms can be handled more conveniently in $C^2$ estimates.
Compared to Lu-Tsai's work in case $\lambda_n\leq -A$, we optimize the coeffcient of the term $\frac{ F^{11}\xi_{1}^2}{\lambda_1}$ to $1+\delta_0$ for some small $\delta_0$, which is crucial in proving the rigidity theorem \ref{main-4}. 
\end{remark}

The rest of the paper is organized as follows. In Section 2, we start with some
preliminaries. The proofs of Lemma \ref{pro-55-33} and Theorem \ref{main-2} are respectively given in section 3 and section 4. In the last section, a rigidity theorems of the equations \eqref{digity-00} in case $k=n-1$ is established.


\section{Preliminaries}

Let $\lambda=(\lambda_1,\cdots,\lambda_n)\in\mathbb{R}^n$, we recall
the definition of elementary symmetric function for $1\leq k\leq n$,
\begin{equation}\label{g}
\sigma_k(\lambda)= \sum _{1 \le i_1 < i_2 <\cdots<i_k\leq
n}\lambda_{i_1}\lambda_{i_2}\cdots\lambda_{i_k}.
\end{equation}
We also set $\sigma_0=1$ and $\sigma_k=0$ for $k>n$ or $k<0$. Recall
that the Garding's cone is defined as
\begin{equation}\label{g-k}
\Gamma_k  = \{ \lambda  \in \mathbb{R}^n :\sigma _i (\lambda ) >
0,\forall~ 1 \le i \le k\}.
\end{equation}
We denote $\sigma_{k-1}(\lambda|i)=\frac{\partial
\sigma_k}{\partial \lambda_i}$. Then, we list some properties of
$\sigma_k$ which will be used later.

\begin{proposition}\label{sigma}
Let $\lambda=(\lambda_1,\cdots,\lambda_n)\in\mathbb{R}^n$ and $1\leq
k\leq n$, then we have

\begin{enumerate}
\item [{\em(1)}] $$\Gamma_1\supset \Gamma_2\supset \cdot\cdot\cdot\supset
\Gamma_n.$$
\item [{\em(2)}]    $$\sigma_k(\lambda)=\sigma_k(\lambda|i)
+\lambda_i\sigma_{k-1}(\lambda|i)~ \mbox{for}~1\leq i\leq n.$$
\item [{\em(3)}] 
$$\sum_{i=1}^{n}\sigma_{k-1}(\lambda|i)=(n-k+1)\sigma_{k-1}(\lambda).$$
\item [{\em(4)}] $\sum_{i=1}^{n}\frac{\partial \sigma_{k}^{\frac{1}{k}}}
{\partial \lambda_i}\geq [C^k_n]^{\frac{1}{k}}$ for
$\lambda \in \Gamma_{k}$.
\item [{\em(5)}] $\sigma_k^{\frac{1}{k}}$ is concave in $\Gamma_k$.
\end{enumerate}
\end{proposition}

\begin{proof}
All the above properties are well known. For example, see Chapter XV in
\cite{Li96} or \cite{Hui99} for proofs of (1), (2) and
(3); see Lemma 2.2.19 in \cite{Ger06} for the proof of (4); see
\cite{CNS85} and \cite{Li96} for the proof of (5).
\end{proof}

\begin{proposition}\label{Nm-11}
For $\lambda=(\lambda_1, \cdots, \lambda_n) \in \Gamma_k$ with $\lambda_1\geq \lambda_2 \geq \cdots \geq \lambda_n$.
\begin{enumerate}
\item [{\em(\romannumeral1)}]
$$0<\sigma_{k-1}(\lambda|1)\leq \sigma_{k-1}(\lambda|2)\leq
\cdot\cdot\cdot\leq \sigma_{k-1}(\lambda|n).$$
\item [{\em(\romannumeral2)}] If $\lambda_i \leq 0$, then we have
$$-\lambda_i\leq \frac{n-k}{k} \lambda_1.$$
\item[{\em(\romannumeral3)}]  For any  $1\leq l<k$, we have
$$\sigma_l(\lambda) \geq C(n, l)\lambda_1\cdots\lambda_l.$$
\item[{\em(\romannumeral4)}] There exists a constant depending only on $n$ and $k$
 such that
$$\lambda_1\sigma_{k-1}(\lambda|1)\geq C \sigma_k(\lambda).$$
\item[{\em(\romannumeral5)}]
$$\sum_i \lambda_i^2 \sigma_{k-1} (\lambda|i)\geq \frac{k}{n} \sigma_1(\lambda) \sigma_k(\lambda).$$
\item[{\em(\romannumeral6)}] The number of possible negative entries of $\lambda$ is at most $n-k$ and 
$$\lambda_k+\lambda_{k+1}+\cdots+\lambda_n >0, \quad |\lambda_i| \leq n \lambda_k~\mbox{for any}~i>k.$$
\item[{\em(\romannumeral7)}] $$\sigma_k(\lambda)\leq C_n^k \lambda_1 \cdots \lambda_k.$$
\end{enumerate}
\end{proposition}
\begin{proof}
See \cite[Lemma 2.2]{lu23} for  \em(\romannumeral1)-\em(\romannumeral5); see \cite{RW-20} for \em(\romannumeral6); see Chapter XV in \cite{Li96} for \em(\romannumeral7).
\end{proof}

The generalized Newton-MacLaurin inequality is as follows, which
will be used all the time.
\begin{proposition}\label{NM}
For $\lambda \in \Gamma_m$ and $m > l \geq 0$, $ r > s \geq 0$, $m
\geq r$, $l \geq s$, we have
\begin{align}
\Bigg[\frac{{\sigma _m (\lambda )}/{C_n^m }}{{\sigma _l (\lambda
)}/{C_n^l }}\Bigg]^{\frac{1}{m-l}} \le \Bigg[\frac{{\sigma _r
(\lambda )}/{C_n^r }}{{\sigma _s (\lambda )}/{C_n^s
}}\Bigg]^{\frac{1}{r-s}}. \notag
\end{align}
\end{proposition}
\begin{proof}
See \cite{S05}.
\end{proof}

Next, we list the following well-known results.

\begin{lemma}\label{Gao0}
If $W=(w_{ij})$ is a symmetric real matrix, $\lambda_i=\lambda_i(W)$
is one of the eigenvalues ($i = 1, \cdots, n$) and
$F=F(W)=f(\lambda(W))$ is a symmetric function of $\lambda_1, \cdots,
\lambda_n$, then for any real symmetric matrix $A= (a_{ij})$, we
have the following formulas:
\begin{eqnarray}\label{Pre-2}
\frac{\partial^2 F}{\partial w_{ij}\partial
w_{st}}a_{ij}a_{st}=\frac{\partial^2 f}{\partial\lambda_p
\partial\lambda_q}a_{pp}a_{qq}+2\sum_{p<q}\frac{\frac{\partial
f}{\partial \lambda_p}-\frac{\partial f}{\partial
\lambda_q}}{\lambda_p-\lambda_q}a^{2}_{pq}.
\end{eqnarray}
Moreover, if $f$ is concave and $\lambda_1\geq \lambda_2\geq \cdots
\geq \lambda_n$, we have
\begin{eqnarray}\label{Pre-3}
\frac{\partial f}{\partial \lambda_1}(\lambda)\leq \frac{\partial
f}{\partial \lambda_2}(\lambda)\leq \cdots \leq \frac{\partial
f}{\partial \lambda_n}(\lambda).
\end{eqnarray}
\end{lemma}
\begin{proof}
See Lemma 3.2 in   \cite{And94} for the proof of
\eqref{Pre-2} and Lemma 2.2 in \cite{And94} for the proof
\eqref{Pre-3}.
\end{proof}

\begin{lemma}\label{pro-54}
Let $W={W_{ij}}$ be an $n\times n$ symmetric matrix and $\lambda(W)=(\lambda_{1},\lambda_{2},\cdots,\lambda_{n})$ be the eigenvalues of the symmetric matrix $W$. Suppose that $W={W_{ij}}$ is diagonal and $\lambda_{i}=W_{ii}$, then we have
\begin{eqnarray}
\frac{\partial{\lambda_{i}}}{\partial{W_{ii}}}=1,\quad \frac{\partial{\lambda_{k}}}{\partial{W_{ij}}}=0\quad \mbox{otherwise}.
\end{eqnarray}
\begin{eqnarray}
\frac{\partial^{2} \lambda_{i}}{\partial{W_{ij}}\partial{W_{ji}}}=\frac{1}{\lambda_{i}-\lambda_{j}}\quad \mbox{for}~ i\neq j \quad and \quad \lambda_{i}\neq \lambda_{j}.
\end{eqnarray}
\begin{eqnarray}
\frac{\partial^{2} \lambda_{i}}{\partial{W_{kl}}\partial{W_{pq}}}=0 \quad \mbox{otherwise}.
\end{eqnarray}
\end{lemma}


\section{A concavity inequality for $(n-1)$-Hessian operator}

In this section, we prove the proof of the crucial  Lemma \ref{pro-55-33}. 
We need the following lemma which is a slight improvement of Lemma 3.1 in \cite{lu24}.

\begin{lemma}\label{pro-55}
Let $n\geq 3$, $F=\sigma_{n-1}$ and $\lambda=(\lambda_1, \cdots, \lambda_n)\in \Gamma_{n-1}$ with
$$\lambda_1=\cdots=\lambda_m>\lambda_{m+1} \geq \cdots \geq \lambda_n.$$
Then there exists  constant $A>1$ and small constant $\delta_0>0$ depending on $n, A$ and $\max F$ such that if $\lambda_n \leq -A$, then we have 
\begin{eqnarray*}
-\sum_{p\neq q} F^{pp, qq} \xi_{p} \xi_{q}+\frac{\left(\sum_i F^{ii} \xi_{i}\right)^2}{F}+2 \sum_{i>m} \frac{F^{ii} \xi_{i}^2}{\lambda_1-\lambda_i} \geq (1+\delta_0) \frac{ F^{11}\xi_{1}^2}{\lambda_1}
\end{eqnarray*}
for any $\xi=(\xi_1, \cdots, \xi_n)\in \mathbb{R}^n$ with $\xi_i=0$ for $1<i\leq m$.
\end{lemma}

\begin{proof}
Denote $\Lambda=-\sigma_n$, then we have
\begin{eqnarray}\label{pro-5552-3}
\Lambda =\lambda_1\cdots\lambda_{n-1}\cdot (-\lambda_n)\geq A^{n-1} \lambda_1
\end{eqnarray}
since $\lambda_i+\lambda_n>0$ for $i=1\cdots n-1$.
By Proposition \ref{sigma} and the definition of $\Lambda$, it is esay to show that
$$F^{ii}=\frac{F}{\lambda_{i}} +\frac{\Lambda}{\lambda_i^2},$$
and
$$F^{ii, jj}=\frac{F}{\lambda_{i}\lambda_{j}} +\frac{\Lambda(\lambda_i+\lambda_j)}{\lambda_i^2\lambda_j^2} \quad \mbox{for}~i\neq j.$$
Hence, we have
\begin{eqnarray}\label{pro-55-3}
&&-\sum_{p\neq q} F^{pp, qq}\xi_{p}\xi_{q} +\frac{\left(\sum_i F^{ii} \xi_{i}\right)^2}{F}\nonumber\\
&=&-\sum_{p\neq q} \left(\frac{F}{\lambda_{p}\lambda_{q}} +\frac{\Lambda(\lambda_p+\lambda_q)}{\lambda_p^2\lambda_q^2}\right) \xi_p \xi_q+\sum_{p\neq q} \left(\frac{F}{\lambda_{p}} +\frac{\Lambda}{\lambda_p^2}\right)  \left(\frac{F}{\lambda_{q}} +\frac{\Lambda}{\lambda_q^2}\right) \frac{\xi_p \xi_q}{F}\nonumber\\
&&+ \sum_i \left(\frac{F}{\lambda_{i}} +\frac{\Lambda}{\lambda_i^2}\right)^2 \frac{\xi_i^2}{F}\nonumber\\
&=&\frac{\Lambda^2}{F} \sum_{p,q}  \frac{\xi_p\xi_q}{\lambda_p^2\lambda_q^2}+F\sum_i \frac{\xi_i^2}{\lambda_i^2}+2\Lambda \sum_i \frac{\xi_i^2}{\lambda_i^3},
\end{eqnarray}
and
\begin{eqnarray}\label{pro-5555-3}
&&2 \sum_{i>m} \frac{F^{ii} \xi_{i}^2}{\lambda_1-\lambda_i} - (1+\delta_0) \frac{ F^{11}\xi_{1}^2}{\lambda_1}\\
&=& F \left(2\sum_{i>m} \frac{ \xi_{i}^2}{(\lambda_1-\lambda_i)\lambda_i}- (1+\delta_0) \frac{\xi_1^2}{\lambda_1^2}\right) +\Lambda \left( 2\sum_{i>m} \frac{ \xi_{i}^2}{(\lambda_1-\lambda_i)\lambda_i^2}- (1+\delta_0) \frac{\xi_1^2}{\lambda_1^3}\right).\nonumber
\end{eqnarray}
Combining with \eqref{pro-5552-3},  \eqref{pro-55-3}, \eqref{pro-5555-3} and the assumption $\xi_i=0$ for $1<i\leq m$, we have 
\begin{eqnarray*}
&&-\sum_{p\neq q} F^{pp, qq} \xi_{p} \xi_{q}+\frac{\left(\sum_i F^{ii} \xi_{i}\right)^2}{F}+2 \sum_{i>m} \frac{F^{ii} \xi_{i}^2}{\lambda_1-\lambda_i} - (1+\delta_0) \frac{ F^{11}\xi_{1}^2}{\lambda_1}\\
&=& F \left(\sum_i \frac{\xi_i^2}{\lambda_i^2}+2\sum_{i>m} \frac{ \xi_{i}^2}{(\lambda_1-\lambda_i)\lambda_i}- (1+\delta_0) \frac{\xi_1^2}{\lambda_1^2}\right)\\
&&+ \Lambda \left(\frac{\Lambda}{F} \sum_{p,q}  \frac{\xi_p\xi_q}{\lambda_p^2\lambda_q^2}+ \sum_i \frac{2\xi_i^2}{\lambda_i^3}+ 2\sum_{i>m} \frac{ \xi_{i}^2}{(\lambda_1-\lambda_i)\lambda_i^2}- (1+\delta_0) \frac{\xi_1^2}{\lambda_1^3}\right)\\
&\geq& \Lambda \left(\frac{\Lambda}{F} \sum_{p,q }  \frac{\xi_p\xi_q}{\lambda_p^2\lambda_q^2}+ \sum_i \frac{2\xi_i^2}{\lambda_i^3}+ 2\sum_{i>m} \frac{ \xi_{i}^2}{(\lambda_1-\lambda_i)\lambda_i^2}- (1+\delta_0) \frac{\xi_1^2}{\lambda_1^3}-\delta_0 \frac{F\xi_1^2}{\Lambda\lambda_1^2}\right)\\
&\geq& \Lambda \left(\frac{\Lambda}{F} \sum_{p,q \in\{1, m+1, \cdots, n\}}   \frac{\xi_p\xi_q}{\lambda_p^2\lambda_q^2}+ \sum_{i>m}  \frac{ 2\xi_{i}^2 \lambda_1}{(\lambda_1-\lambda_i)\lambda_i^3}  +(1-\delta_0-\frac{\delta_0F}{A^{n-1}}) \frac{\xi_1^2}{\lambda_1^3}\right)\\
&=&  \Lambda \cdot \eta (y^Ty+D)\eta^T,
\end{eqnarray*}
where $y=\sqrt{\frac{\Lambda}{F}} (1,1, \cdots, 1) \in \mathbb{R}^{n-m+1}$,
$\eta=(\eta_1\cdots, \eta_n) \in \mathbb{R}^{n-m+1}$ with 
$$\eta_1=\frac{\xi_1}{\lambda_1^2}, \quad \eta_i=\frac{\xi_{m+i-1}}{\lambda_{m+i-1}^2} ~\mbox{for}~i=2,\cdots n-m+1,$$
 and $D=\mbox{diag}(d_1, \cdots, d_{n-m+1})$ is a $(n-m+1)\times (n-m+1)$ diagonal matrix  with
$$d_i=\begin{cases} \left(1-\delta_0-\frac{\delta_0F}{A^{n-1}}\right)\lambda_1, \quad i=1;\\
\frac{2 \lambda_1\lambda_{m+i-1}}{(\lambda_1-\lambda_{m+i-1})}, \quad \quad \quad\quad 2\leq i\leq n-m+1.
\end{cases}$$
We shall have established the lemma if we prove that $y^Ty+D$ is a positive definite. Assuming that $A> (3 \max F+1)^{\frac{1}{n-1}}$, then we can choose  $\delta_0$ sufficiently small such that  $\delta_0+\frac{\delta_0 \max F}{A^{n-1}}<\frac{1}{4}$. It implies that the 1st, 2nd,  $\cdots, (n-m)$th  order leading principle minor of $y^Ty+D$ is greater than $0$ since $d_i>0$ for $1\leq i\leq n-m$. Hence we only need to show $\det (y^Ty+D)>0$. In fact
\begin{eqnarray}\label{pro-55-1314}
\det (y^Ty+D) =\det D\cdot(1+ yD^{-1}y^T)= \det D \left(1+ \frac{\Lambda}{F} \sum_i \frac{1}{d_i}\right).
\end{eqnarray}
On the one hand $\det D<0$. On the other hand,
\begin{eqnarray}\label{pro-55443-1}
1+ \frac{\Lambda}{F} \sum_i \frac{1}{d_i}& = &1+  \frac{\Lambda}{F} \left(\frac{1}{ \left(1-\delta_0-\frac{\delta_0F}{A^{n-1}}\right)\lambda_1}+\sum_{i> m} \frac{\lambda_1-\lambda_i}{2\lambda_1\lambda_i}\right)\\
&\leq& 1+  \frac{\Lambda}{F} \left(\frac{4}{3} \frac{1}{\lambda_1}-\frac{n}{2}\frac{1}{\lambda_1} +\sum_{i} \frac{1}{2\lambda_i}\right)\nonumber\\
&=&1+\left(\frac{4}{3}-\frac{n}{2}\right) \frac{\Lambda}{F\lambda_1}+\frac{\Lambda F}{2F \sigma_n} \nonumber\\
&\leq&\frac{1}{2}-\frac{1}{6}\frac{\Lambda}{F\lambda_1}\nonumber\\
&\leq& \frac{1}{2}-\frac{1}{6}\frac{A^{n-1}}{\max F}\nonumber\\
&<& 0.\nonumber
\end{eqnarray}
We have used the assumption that  $\delta_0+\frac{\delta_0 \max F}{A^{n-1}}<\frac{1}{4}$ and $A> (3 \max F+1)^{\frac{1}{n-1}}$ in the second and last line.
Thus $\det (y^Ty+D)>0$ and hence the lemma is proved.
\end{proof}

\begin{lemma}\label{pro-55-1}
Let $n\geq 3$, $\lambda=(\lambda_1, \cdots, \lambda_n)\in \Gamma_{k}$ with
$$\lambda_1=\cdots=\lambda_m>\lambda_{m+1} \geq \cdots \geq \lambda_n \geq -A$$
for a constant $A>0$. 
Then there exists  constant $C>1$ depending on $n, k, \sigma_k, A$ and small constant $\delta_0>0$  depending on $k$ such that if $\lambda_1 \geq C$, then we have 
\begin{eqnarray}\label{pro-55443-2}
-\sum_{p\neq q} \sigma_k^{pp, qq} \xi_{p} \xi_{q}+ K\frac{\left(\sum_i \sigma_k^{ii} \xi_{i}\right)^2}{\sigma_k}+2 \sum_{i>m} \frac{\sigma_k^{ii} \xi_{i}^2}{\lambda_1-\lambda_i} \geq (1+\delta_0) \frac{ \sigma_k^{11}\xi_{1}^2}{\lambda_1}
\end{eqnarray}
for some sufficient large $K>0$ (depending on $\delta_0$), where $\xi=(\xi_1, \cdots, \xi_n)\in \mathbb{R}^n$ with $\xi_i=0$ for $1<i \leq m$.
\end{lemma}

\begin{proof}
The proof is similar to Lemma 1.1 in \cite{Zhang24}, we mainly outline the changes. By (3.2), (3.25) and (3.26) in \cite{Zhang24}, we have  for $K\geq(k+1)^2$,
\begin{eqnarray*}
&&-\sum_{p\neq q} \frac{\sigma_k^{pp, qq} \xi_{p} \xi_{q}}{\sigma_k}+ K\frac{\left(\sum_i \sigma_k^{ii} \xi_{i}\right)^2}{\sigma_k^2}+2 \sum_{i>m} \frac{\sigma_k^{ii} \xi_{i}^2}{(\lambda_1-\lambda_i) \sigma_k} \nonumber\\
&\geq& -\frac{\partial_{\xi}^2 q_k}{ q_k}+ \left(1+\frac{(1-\epsilon)}{k}-\epsilon\right) \left(1-\frac{C(A)}{\lambda_1}\right)^2 \frac{\xi^2_1}{\lambda_1^2}+ \left(2- \frac{C(A)}{\lambda_1}- \frac{C_1}{\epsilon \lambda^{\frac{1}{k-1}}_1}\right) \sum_{i>m} \frac{\sigma_k^{ii}\xi_i^2}{\lambda_1 \sigma_k}, \nonumber
\end{eqnarray*}
where $q_k:= \frac{\sigma_k}{\sigma_{k-1}}$, $C(A)$ and $C$ are constants. We choose $\epsilon=\frac{1}{(k+1)^2}$ and assume that 
$$\lambda_1 >\max \left\{ \frac{C(A)}{1-\sqrt{\frac{k+2}{k+3}}}, (1+\frac{C(A)}{C_1})^2 \right\}.$$
 Then the above inequality  becomes
\begin{eqnarray}\label{pro-55443-3}
&&-\sum_{p\neq q} \frac{\sigma_k^{pp, qq} \xi_{p} \xi_{q}}{\sigma_k}+ K\frac{\left(\sum_i \sigma_k^{ii} \xi_{i}\right)^2}{\sigma_k^2}+2 \sum_{i>m} \frac{\sigma_k^{ii} \xi_{i}^2}{(\lambda_1-\lambda_i) \sigma_k}\\
&\geq& -\frac{\partial_{\xi}^2 q_k}{ q_k}+\frac{(k+2)^2}{(k+1)(k+3)}  \frac{\xi^2_1}{\lambda_1^2}+ \left(2- 2(k+1)^2\frac{C_1}{ \lambda^{\frac{1}{k-1}}_1}\right) \sum_{i>m} \frac{\sigma_K^{ii}\xi^2}{\lambda_1 \sigma_k}, \nonumber
\end{eqnarray}
If
$$\sigma_k(\lambda|1)\geq -\frac{\sigma_k}{2(k+2)^2-1},$$
the inequality \eqref{pro-55443-2} holds  by choosing $\delta_0=\frac{1}{2(k+1)(k+3)}$ and assuming $\lambda_1 \geq (C_1(k+1)^2)^{k-1}$.

If 
$$\sigma_k(\lambda|1)\leq -\frac{\sigma_k}{2(k+2)^2-1},$$
analysis similar to  the proof of Lemma 1.1 in \cite{Zhang24} shows that
\begin{eqnarray*}
&&-\sum_{p\neq q} \frac{\sigma_k^{pp, qq} \xi_{p} \xi_{q}}{\sigma_k}+ K\frac{\left(\sum_i \sigma_k^{ii} \xi_{i}\right)^2}{\sigma_k^2}+2 \sum_{i>m} \frac{\sigma_k^{ii} \xi_{i}^2}{(\lambda_1-\lambda_i) \sigma_k}\geq (1+\delta_0)\frac{\sigma_k^{11}\xi^2_1}{\sigma_k \lambda_1^2}.
\end{eqnarray*}
where $\delta_0=\min \left\{\frac{1}{15}, \frac{1}{(k+1)(k+3)} \right\}$.
Hence the proof is completed. 
\end{proof}

\begin{proof}[\textbf{Proof of Lemma \ref{pro-55-33}}]
Combining with Lemma \ref{pro-55} and \ref{pro-55-1}, we obtain the
 Lemma \ref{pro-55-33}.
\end{proof}


\section{Pogorelov type $C^2$ estimates for $(n-1)$-equations}

In this section, we will use the idea in \cite{Chu20, Chen20} to
give the proof of Theorem \ref{main-2}. For convenience, we introduce
the following notations
\begin{eqnarray*}
F(\nabla^2u)=\sigma_{n-1}(\nabla^2u),
\quad F^{ij}=\frac{\partial F}{\partial u_{ij}}, \quad F^{ij, r
s}=\frac{\partial^2 F}{\partial u_{ij}\partial u_{rs}}.
\end{eqnarray*}
We assume $u\in C^4(\Omega)\cap C^2(\overline{\Omega})$ is a $(n-1)$-convex
solution of the Hessian equation \eqref{Eq-n-1} with $\lambda(\nabla^2u)\in \Gamma_{n-1}$.

\begin{proof}
Note that we can assume $u<0$ in $\Omega$ by the maximum principle. We consider the following test function,
\begin{equation*}
P(x)=\ln
\lambda_{1}+ \beta \ln (-u)+\frac{B}{2}|\nabla u|^2,
\end{equation*}
where $\lambda_{1}$ is the biggest eigenvalue of the Hessian
matrix $\nabla^2 u$,  $\beta$ and $B$ are constants which will be
determined later. Suppose $P$ attains its maximum
value in $\Omega$ at $x_0$. Rotating the coordinates, we diagonal
the matrix $\nabla^2 u=(u_{ij})$. Without loss of generality, we may assume $\lambda_1(x_0)$ has multiplicity $m$, then
\begin{eqnarray*}
u_{ij}=u_{ii}\delta_{ij}, \quad \lambda_i=u_{ii}, \quad \lambda_{1}=\lambda_2=\cdots=\lambda_m\geq
\lambda_{m+1}\geq \cdots \geq \lambda_{n}\quad \mbox{at}~x_0.
\end{eqnarray*}
Then, by \eqref{Pre-3} we obtain
\begin{eqnarray}\label{F-or}
0<F^{11}= F^{22}=\cdots =F^{mm}\leq F^{m+1, m+1}\leq \cdots \leq F^{nn} \quad \mbox{at}~x_0.
\end{eqnarray}
By Lemma 5 in \cite{BCD-17}, we have 
\begin{equation}\label{90-9873}
\delta_{kl} \cdot \lambda_{1,i}= u_{kli}~\quad~\mbox{for}~1\leq k, l\leq m,
\end{equation}
and
\begin{equation}\label{90-9874}
\lambda_{1,ii} \geq u_{11ii}+2 \sum_{p>m} \frac{u_{1pi}^2}{\lambda_1-\lambda_p}
\end{equation}
in the viscosity sense.

In the following, we will do a standard computation at $x_0$.
Note that in the following estimates the letter ``$C$" denotes a generic constant
which is allowed to depend only on the known data of the problem, i.e.
$n, k,  \sup_{\Omega}|u|$, $\sup_{\Omega}|\nabla u|, |\psi|_{C^2}, \inf \psi$  and which may change from line to line.

 Differentiating $P$ at
$x_0$ twice, we get
\begin{equation}\label{102803}
\frac{\lambda_{1,i}}{\lambda_1}+\frac{\beta u_i}{u} +Bu_i u_{ii}=0,
\end{equation}
and
\begin{equation}\label{102804}
\frac{\beta u_{ii}}{u} - \frac{\beta u_i^2}{u^2} +
\frac{\lambda_{1, ii}}{\lambda_1} -\frac{\lambda_{1,i}^2}{\lambda_{1}^2} +B
\sum_{j=1}^{n} u_j u_{jii} + B u^2_{ii}\leq 0.
\end{equation}
According to \eqref{90-9873} and \eqref{90-9874}, we have at $x_0$
\begin{eqnarray}\label{TPP-0}
\frac{\lambda_{1,i}}{\lambda_1}=\frac{u_{11i}}{\lambda_1}, \quad \frac{\lambda_{1, ii}}{\lambda_1}\geq \frac{u_{11ii}}{\lambda_1}+2 \sum_{p>m} \frac{u_{1pi}^2}{\lambda_1(\lambda_1-\lambda_{p})}.
\end{eqnarray}
Thus, at $x_0$
\begin{eqnarray}\label{TP}
0&\geq& F^{ii} P_{ii}\nonumber\\
&\geq&\frac{ (n-1)\beta\psi}{u} -
\frac{\beta F^{ii}u_i^2}{u^2} + \frac{F^{ii}u_{11ii}}{\lambda_1}+2 \sum_{p>m} \frac{F^{ii}u_{1pi}^2}{\lambda_1(\lambda_1-\lambda_{p})}
-\frac{F^{ii}u_{11i}^2}{\lambda_{1}^2}\nonumber\\&&+B\sum_{j=1}^{n}  u_j
F^{ii}u_{iij} + B F^{ii}u^2_{ii}\nonumber\\
&\geq& \frac{F^{ii}u_{11ii}}{\lambda_1}+2 \sum_{p>m} \frac{F^{ii}u_{1pi}^2}{\lambda_1(\lambda_1-\lambda_{p})}+B\sum_{j=1}^{n}  u_j
F^{ii}u_{iij}+ BF^{ii}u^2_{ii}+\frac{ (n-1)\beta\psi}{u} \nonumber\\
&&-(1+\frac{2}{\beta})\frac{F^{ii}u_{11i}^2}{\lambda_{1}^2} - \frac{2B^2}{\beta}F^{ii}u_i^2u^2_{ii}.
\end{eqnarray}

 Rewrite equation \eqref{Eq} as
\begin{equation}\label{Eq-N}
F(\nabla^2 u)=\psi,
\end{equation}
differentiating equation \eqref{Eq-N} once gives
\begin{equation}\label{2}
F^{ii} u_{iij}=\psi_j=\psi_{x_j}+\psi_z u_j+\frac{\partial \psi}{\partial u_{j}} u_{jj}\leq C(1+\lambda_1).
\end{equation}
Differentiating equation \eqref{Eq-N} twice gives
\begin{equation}\label{102802}
F^{ij,rs}u_{ij1}u_{rs1}+F^{ii}u_{ii11}=\psi_{11}.
\end{equation}
Now we estimate $\psi_{11}$
\begin{eqnarray*}
\psi_{11}&\geq& -C(1+\lambda_1+\lambda^{2}_{1})+
\frac{\partial \psi}{\partial u_{i}}u_{11i}.
\end{eqnarray*}
Then, applying Lemma \ref{Gao0}, it follows at $x_0$
\begin{eqnarray}\label{3}
\frac{F^{ii}u_{ii 11}}{\lambda_1}&\geq&- \frac{1}{\lambda_1}F^{ij,rs}u_{ij1}u_{rs1}-C(1+\lambda_1)+
\frac{\partial \psi}{\partial u_{i}}\frac{u_{11i}}{\lambda_1}\nonumber\\&=&
- \frac{1}{\lambda_1}\sum_{p \neq q}F^{pp,qq}u_{pp1}u_{qq1}- \frac{2}{\lambda_1}\sum_{p<q} \frac{F^{pp}-F^{qq}}{\lambda_p-\lambda_q}  u_{pq1}^2-C(1+\lambda_1)+
\frac{\partial \psi}{\partial u_{i}}\frac{u_{11i}}{\lambda_1}\nonumber\\
&\geq&- \frac{1}{\lambda_1}\sum_{p \neq q}F^{pp,qq}u_{pp1}u_{qq1}+\sum_{i>m} \frac{2(F^{ii}-F^{11})u_{11i}^2}{(\lambda_1-\lambda_i)\lambda_1} -C(1+\lambda_1)+
\frac{\partial \psi}{\partial u_{i}}\frac{u_{11i}}{\lambda_1}.
\end{eqnarray}
Using \eqref{2} and \eqref{102803}, we have,
\begin{eqnarray}\label{3-2-2}
\frac{\partial \psi}{\partial u_{i}}\frac{u_{11i}}{\lambda_1}+B\sum_{j=1}^{n}  u_j
F^{ii}u_{iij} \geq -C(B+\frac{\beta}{-u}).
\end{eqnarray}

Plugging \eqref{2}, \eqref{3} and \eqref{3-2-2} into \eqref{TP},
assuming $u_{11}(x_0)\geq 1$, we have at $x_0$,
\begin{eqnarray}\label{TP-1}
0&\geq&- \frac{1}{\lambda_1}\sum_{p \neq q}F^{pp,qq}u_{pp1}u_{qq1}+2 \sum_i\sum_{p>m} \frac{F^{ii}u_{1pi}^2}{\lambda_1(\lambda_1-\lambda_{p})}+\sum_{i>m} \frac{2(F^{ii}-F^{11})u_{11i}^2}{(\lambda_1-\lambda_i)\lambda_1} \nonumber\\
&&-(1+\frac{2}{\beta})\frac{F^{ii}u_{11i}^2}{\lambda_{1}^2}+ (B-2\frac{B^2C}{\beta})
F^{ii}\lambda^2_{i}-C(B+\frac{\beta}{-u}+\lambda_{1})\nonumber\\
&\geq&- \frac{1}{\lambda_1}\sum_{p \neq q}F^{pp,qq}u_{pp1}u_{qq1}+2 \sum_{p>m} \frac{F^{pp}u_{1pp}^2}{\lambda_1(\lambda_1-\lambda_{p})}+2\sum_{p>m} \frac{F^{11}u_{11p}^2}{\lambda_1(\lambda_1-\lambda_{p})} \nonumber\\
&&+\sum_{p>m} \frac{2(F^{pp}-F^{11})u_{11p}^2}{(\lambda_1-\lambda_p)\lambda_1}
-(1+\frac{2}{\beta})\frac{F^{ii}u_{11i}^2}{\lambda_{1}^2}+ (B-\frac{2B^2C}{\beta})
F^{ii}\lambda^2_{i}\nonumber\\
&&-C(B+\frac{\beta}{-u}+\lambda_{1}),
\end{eqnarray}
since
\begin{eqnarray*}
2 \sum_i\sum_{p>m} \frac{F^{ii}u_{1pi}^2}{\lambda_1(\lambda_1-\lambda_{p})} &\geq&2 \sum_{p>m} \frac{F^{pp}u_{1pp}^2}{\lambda_1(\lambda_1-\lambda_{p})}+ 2\sum_{p>m} \frac{F^{11}u_{11p}^2}{\lambda_1(\lambda_1-\lambda_{p})}.
\end{eqnarray*}

 Then we obtain the following lemma:

\begin{lemma}\label{claim-21}
Assuming that $\beta>\frac{2n}{n-2}$, then we have at $x_0$
\begin{eqnarray*}
2\sum_{p>m} \frac{F^{11}u_{11p}^2}{\lambda_1(\lambda_1-\lambda_{p})}+\sum_{p>m} \frac{2(F^{pp}-F^{11})u_{11p}^2}{(\lambda_1-\lambda_p)\lambda_1}-(1+\frac{2}{\beta}) \sum_{p>1}\frac{F^{pp}u_{11p}^2}{\lambda_{1}^2}\geq 0.
\end{eqnarray*}
\end{lemma}
\begin{proof}
By \eqref{90-9873}, we have 
\begin{equation}
u_{11i}=u_{1i1}=\delta_{1i}\cdot \lambda_{1,i}=1, \quad~\mbox{for}~1<i\leq m.
\end{equation}
By direct calculation, we obtain
\begin{eqnarray}
&&2\sum_{p>m} \frac{F^{11}u_{11p}^2}{\lambda_1(\lambda_1-\lambda_{p})}+\sum_{p>m} \frac{2(F^{pp}-F^{11})u_{11p}^2}{(\lambda_1-\lambda_p)\lambda_1}-(1+\frac{2}{\beta}) \sum_{p>1}\frac{F^{pp}u_{11p}^2}{\lambda_{1}^2}\nonumber\\
&\geq&\sum_{p>m}\frac{F^{pp}u_{11p}^2}{\lambda_{1}}\frac{(1-\frac{2}{\beta})\lambda_1+(1+\frac{2}{\beta})\lambda_p}{(\lambda_1-\lambda_p)\lambda_1}\nonumber\\
&=&(1+\frac{2}{\beta})\sum_{p>m}\frac{F^{pp} u_{11p}^2}{\lambda_1^2}\frac{\frac{\beta-2}{\beta+2}\lambda_1+\lambda_p}{\lambda_1-\lambda_{p}} \nonumber
\end{eqnarray}
On the one hand, 
$$\frac{\frac{\beta-2}{\beta+2}\lambda_1+\lambda_p}{\lambda_1-\lambda_{p}}\geq 0, \quad \mbox{if}~\lambda_p\geq 0.$$
  When $\lambda_p<0$, then
 $$\frac{\frac{\beta-2}{\beta+2}\lambda_1+\lambda_p}{\lambda_1-\lambda_{p}}=-1+ \frac{1+\frac{\beta-2}{\beta+2}}{1-\frac{\lambda_{p}}{\lambda_1}}\geq 0,$$
 the last inequality comes from Proposition \ref{Nm-11}.
\end{proof}

Using \eqref{TP-1} and Lemma \ref{claim-21}, we have
\begin{eqnarray}\label{TP-121}
0&\geq&- \frac{1}{\lambda_1}\sum_{p \neq q}F^{pp,qq}u_{pp1}u_{qq1}+2 \sum_{p>m} \frac{F^{pp}u_{1pp}^2}{\lambda_1(\lambda_1-\lambda_{p})}
-(1+\frac{2}{\beta})\frac{F^{11}u_{111}^2}{\lambda_{1}^2}+ \frac{B}{2}
F^{ii}\lambda^2_{i}\nonumber\\
&&-C(B+\frac{\beta}{-u}+\lambda_{1}),
\end{eqnarray}
if we choose $\beta>\max \{\frac{2n}{n-2}, 4BC \}$.

From Lemma \ref{pro-55-33}, there exists a constant $C, k$ and small constant $\delta_0>0$ such that 
\begin{eqnarray}\label{TP-12931}
&&- \frac{1}{\lambda_1}\sum_{p \neq q}F^{pp,qq}u_{pp1}u_{qq1}+2 \sum_{p>m} \frac{F^{pp}u_{1pp}^2}{\lambda_1(\lambda_1-\lambda_{p})}-(1+\frac{2}{\beta}) \frac{F^{11} u_{111}^2}{\lambda_1^2}\nonumber\\
&\geq& -K\frac{\left(\sum_p F^{pp} u_{pp1}\right)^2}{F \lambda_1}+(\delta_0-\frac{2}{\beta}) \frac{F^{11} u_{111}^2}{\lambda_1^2}\nonumber\\
&>&- CK(1+\lambda_1)
\end{eqnarray}
by assuming $\lambda_1> C$ and choosing $\beta>\{\frac{2n}{n-2}, \frac{2}{\delta_0}, 4BC \}$.
Plugging \eqref{TP-12931} into \eqref{TP-121}, we have
\begin{eqnarray}\label{TP2-126}
0&\geq&-CK(\lambda_1+1)+\frac{B}{2}F^{ii}\lambda^2_{i}-C(B+\frac{\beta}{-u}+\lambda_{1})\nonumber\\
&\geq& \frac{B(n-1)}{2n} \sigma_1\sigma_{n-1}-CK(B+\frac{\beta}{-u}+\lambda_{1})\nonumber\\
&\geq&\frac{B}{8C}  \lambda_1  -C(B+\frac{\beta}{-u}+\lambda_{1}).\nonumber
\end{eqnarray}
It follows that $(-u)^{\beta}\lambda_1 \leq C$ by choosing $a$ large enough. 
Then the theorem is now proved.

\end{proof}


\section{A rigidity theorem for $(n-1)$-Hessian equations}
In this section, we prove Theorem \ref{main-4}. At first we have the following lemmas.

\begin{lemma}\label{rigidity-01}
Let $u\in C^4(\Omega)\cap C^2(\overline{\Omega})$ be a $(n-1)$-convex solution
to the Dirichlet problem of the following Hessian equation  
\begin{equation}\label{Eq-rigi}
\left\{
\begin{aligned}
&\sigma_{n-1}(\nabla^2u)=1 &&in~
\Omega,\\
&u = 0 &&on~\partial \Omega,
\end{aligned}
\right.
\end{equation}
 with $\lambda(\nabla^2u) \in
\Gamma_{n-1}$. 
 Then
\begin{equation*}
(-u)^\beta|\nabla^2 u|(x) \leq C, \quad \forall~ x \in \Omega,
\end{equation*}
for sufficiently large $\beta>0$. Here $C$ and $\beta$ depends only on $n$ and the domain $\Omega$.
\end{lemma}

\begin{proof}

It is obvious that there exist constants $a$ and $b$ depending only on the diameter of the domain $\Omega$ such that
$$\frac{a}{2} x^2-b\leq u\leq 0.$$
Hence, in the following proof, the constant $\beta, C$ can contains $\sup_{\Omega} |u|$. 

We consider the pertrubed quantity $\widetilde{P}$ defined by
\begin{equation*}
P(x)=\ln \lambda_{1}+  \beta\ln (-u)+\frac{1}{2}|x|^2,
\end{equation*}
where $\lambda_{1}$  is the biggest eigenvalue of the Hessian
matrix $\nabla^2 u$. $\beta$ is a constant which will be determined later. 
Suppose $P$ attains its maximum
value in $\Omega$ at $x_0$. Rotating the coordinates, we diagonal
the matrix $\nabla^2 u=(u_{ij})$. Without loss of generality, we may assume $\lambda_1(x_0)$ has multiplicity $m$, then
\begin{eqnarray*}
u_{ij}=u_{ii}\delta_{ij}, \quad \lambda_i=u_{ii}, \quad \lambda_{1}=\lambda_2=\cdots=\lambda_m\geq
\lambda_{m+1}\geq \cdots \geq \lambda_{n}\quad \mbox{at}~x_0.
\end{eqnarray*}
Then, by Lemma 5 in \cite{BCD-17},  we obtain the inequalities \eqref{90-9873} and \eqref{90-9874} in the viscosity sense.

 Differentiating $P$ at
$x_0$ twice, we have
\begin{equation}\label{1102803}
\frac{\lambda_{1,i}}{\lambda_1}+\frac{\beta u_i}{u}+x_i=0,
\end{equation}
and
\begin{equation}\label{1102804}
\frac{\beta u_{ii}}{u} - \frac{\beta u_i^2}{u^2} +
\frac{\lambda_{1, ii}}{\lambda_1} -\frac{\lambda_{1,i}^2}{\lambda_{1}^2} +1\leq 0.
\end{equation}
Analysis similar to that in the proof of Theorem \ref{main-2} show that
\begin{eqnarray}\label{TP-1121}
0&\geq&- \frac{1}{\lambda_1}\sum_{p \neq q}F^{pp,qq}u_{pp1}u_{qq1}+\left(1-\frac{2C}{\beta}\right)\sum_{i=1}^{n}F^{ii}+\frac{\beta (n-1)}{u}\nonumber\\
&&+2\sum_{p>m} \frac{F^{pp}u_{1pp}^2}{\lambda_1(\lambda_1-\lambda_{p})}-(1+\frac{2}{\beta})\frac{F^{11}u_{111}^2}{\lambda_{1}^2}
\end{eqnarray}
if we choose $\beta>\frac{2(n-1)}{n-2}$.

Note that $u_{pp1}=u_{p1p}=0$ for $1<p\leq m$ since \eqref{90-9873}.  According to Lemma \ref{pro-55-33}, there exists a constant C and small constant $\delta_0>0$ such that 
\begin{eqnarray}\label{TP-1293}
&&- \frac{1}{\lambda_1}\sum_{p \neq q}F^{pp,qq}u_{pp1}u_{qq1}+2 \sum_{p>m} \frac{F^{pp}u_{1pp}^2}{\lambda_1(\lambda_1-\lambda_{p})}-(1+\frac{2}{\beta}) \frac{F^{11} u_{111}^2}{\lambda_1^2}\nonumber\\
&\geq& (\delta_0-\frac{2}{\beta}) \frac{F^{11} u_{111}^2}{\lambda_1^2}\nonumber\\
&>&0
\end{eqnarray}
by assuming $\lambda_1> C$ and choosing $\beta>\{\frac{2n-2}{n-2}, \frac{2}{\delta_0} \}$.
Plugging \eqref{TP-1293} into \eqref{TP-1121}, we have
\begin{eqnarray}\label{TP-1125}
0&\geq&(1-\frac{2 C}{\beta})\sum_{i=1}^{n}F^{ii}+\frac{\beta k}{u}.
\end{eqnarray}
We take $\beta= \max \{\frac{2n-2}{n-2}, \frac{2}{\delta_0}, 4C\}+1$, then 
 $$0\geq \frac{\beta (n-1)}{u} + \sigma_{n-2}\geq\frac{\beta (n-1)}{u}+ C\sigma_1^{\frac{1}{n-2}} \sigma_{n-1}^{\frac{n-3}{n-2}}>\frac{\beta (n-1)}{u} +C \lambda_1^{\frac{1}{n-2}}.$$
 Hence we obtain the Lemma.

\end{proof}


\begin{proof}[\textbf{Proof of Theorem \ref{main-4}}]
The proof here is similar to \cite{TW-08, Li16}. So, only a sketch will
be given below. 

Suppose $u$ is an entire solution of the equation  \eqref{digity-00}. For arbitrary positive constant $R>1$. Define
$$\Omega_R:= \{y\in \mathbb{R}^n\mid u(Ry) \leq R^2\},\quad v(y)=\frac{u(Ry)-R^2}{R^2}.$$
We consider the following Dirichlet problem
\begin{equation}\label{Eq-rigi-1}
\left\{
\begin{aligned}
&\sigma_{n-1}(\nabla^2v)=1 &&in~
\Omega_R,\\
&v = 0 &&on~\partial \Omega_R.
\end{aligned}
\right.
\end{equation}
By Lemma \ref{rigidity-01}, it follows that
\begin{equation*}
(-v)^\beta|\nabla^2 v| \leq C,
\end{equation*}
where $C$ and $\beta$ depend only on $n$ and the domain $\Omega_R$. Now using the quadratic growth condition appears in Theorem \ref{main-4}, we can assert that
$$c|Ry|^2-b\leq u(Ry) \leq R^2.$$
Namely 
$$|y|^2 \leq \frac{1+b}{c}.$$
Hence, $\Omega_R$ is bounded and the constant $C$ and $\beta$ become two absolute constants. Consider the domain
$$\Omega'_R=\{ y\mid u(Ry) \leq \frac{R^2}{2} \} \subset \Omega_R.$$
It is easily seen that 
$$ v \leq -\frac{1}{2}, \quad \Delta v \leq 2^{\beta} C\quad \mbox{in}~\Omega'_{R}.$$
Note that $\nabla^2_y v=\nabla^2_x u$. Thus 
$$\nabla u \leq C\quad \mbox{in}~\Omega'_R=\{x\mid u(x) \leq \frac{R^2}{2}
\},$$
where $C$ is absolute constant. Since $R$ is arbitrary, we have the above inequality in whole $\mathbb{R}^n$. Using Evan-Krylov theory, we have 
$$|D^2 u|_{C^{\alpha}(B_R)} \leq C \frac{|D^2 u|_{C^0(B_R)}}{R^{\alpha}} \leq \frac{C}{R^{\alpha}} \longrightarrow 0\quad \mbox{as}~R\rightarrow +\infty.$$
Hence we obtain Theorem  \ref{main-4}.
\end{proof}



\textbf{Conflict of interest statement:}
On behalf of all authors, the corresponding author states that there is no conflict of interest.

\textbf{Data availability statement:}
No datasets were generated or analysed during the current study.

\end{document}